\documentclass[a4paper,10pt]{article}

\usepackage{amsthm}
\usepackage{amsmath}
\usepackage{amssymb}
\usepackage{amsfonts}
\usepackage{fancybox}

\theoremstyle{definition}

\newtheorem{df}{Definition}[section]
\newtheorem{thm}[df]{Theorem}

\newtheorem{cor}[df]{Corollary}
\newtheorem{prop}[df]{Proposition}

\newtheorem{rem}[df]{Remark}

\def\;{{\hspace{0.3ex};\hspace{0.5ex}}}
\def\,{{\hspace{0,3ex},\hspace{0.5ex}}}
\def\({{\hspace{1.2ex}(}}

\def\ph{\varphi}

\def\ra{\rightarrow}

\def\sm{\setminus}
\def\mc{\mathclose}

\def\N{\mathbb{N}}

\def\Inv{{\rm Inv}}
\def\Pol{{\rm Pol}}

\title{Infinitary Theorems in Universal Algebra}
\author{Shohei Izawa\footnote{sa9m02@math.tohoku.ac.jp
}}
\date{}

\begin{document}
\begin{center} {\LARGE Infinitary version of Pol-Inv Galois connection}\end{center}
\begin{center} {\large Shohei Izawa\footnote{sa9m02@math.tohoku.ac.jp}} \end{center}

\begin{abstract}
In this article, we prove infinitary version of one to one correspondence theorem between
clones and relational clones on a fixed possibly infinite set.
We also characterize the relational clone corresponding to the clone of all finitary operations.
By this characterization, we obtain correspondence between finitary clones on a fixed infinite set
and relational clones satisfying some condition.
\end{abstract}

\section{Outline}
In this article, we prove infinitary generalization (Theorem \ref{maincorrespondence}) of
correspondence theorem between clones and relational clones on a fixed set,
proved by \cite{BKKR} and \cite{Gei}. 
So far the author know, only one correspondence result between
clones on an infinite set and other mathematical objects is known.
That is, correspondence between so-called local clones and
local relational clones (\cite{Pos} Theorem 4.1 and 4.2). 
This correspondence only captures particular clones, on the other hand,
correspondence theorem proved in this article captures all clones on a fixed infinite set.

By this correspondence, there is a relational clone that corresponds
to the clone of all finitary operations. This relational clone is characterized as
the set of all ``generalized diagonal relations" (Theorem \ref{finitaryoperation}).
As a corollary, we obtain the correspondence between finitary clones and relational clones
that have all generalized diagonal relations (Corollary \ref{finitaryclones}).

\section{Duality between clones and relational clones}
For this article being self contained, we define concepts appear in the main theorem.

In this article, we use the following notations.
$A\subset B$ denotes the condition $x\in A$ implies $x\in B$,
$A\subsetneq B$ denotes $A\subset B$ and $A\neq B$.
For a set $A$, $|A|$ denotes the cardinality of $A$.

\begin{df}[Clone]
Let $A$ be a set and $\lambda$ be an infinite cardinal.
A set $C$ of $<\mathclose\lambda$-ary operations on $A$,
namely $C\subset \bigcup_{\lambda'<\lambda} A^{A^{\lambda'}}$, is said to be
a $<\mathclose\lambda$-ary (operational) clone on $A$ 
if the following conditions hold:
\begin{enumerate}
\item
For each cardinal $\lambda'<\lambda$ and $i\in\lambda'$, the $i$-th projection $pr_i^{\lambda'}:(a_j)_{j\in\lambda'}\mapsto a_i$
belongs to $C$.
\item
If $\lambda'<\lambda$ and $f,g_i\in C$ for $i\in \lambda''$,
where $\lambda''$ is the arity of $f$, then the composition
\[
f\circ (g_i)_{i\in\lambda'}:(a_j)_{j\in\lambda'}\mapsto 
f(g_i(a_j)_{j\in\lambda'})_{i\in\lambda''}
\]
belongs to $C$.
\end{enumerate}

\end{df}

\begin{df}[Relational Clone]
Let $A$ be a set and $\kappa$ be an infinite cardinal.
A set $R$ of $<\mathclose\kappa$-ary relations on $A$,
that is, $R\subset \bigcup_{\kappa'<\kappa}{\cal P}(A^{\kappa'})$, is said to be
a $<\mc\kappa$-ary relational clone on $A$
if the following conditions hold:

If $\{r_k\}_{k\in K}\subset R$, $r_k\subset A^{\kappa_k}$ and
a relation $r\subset A^{\kappa'}$ is defined by the following form
\[
r=\{(a_{j'})_{j'\in\kappa'}\in A^{\kappa'}\mid 
  \exists (a_{\tilde{j}})_{\tilde{j}\in\tilde{\kappa}} 
  \bigwedge_{u\in U} (a_{f(u,j)})_{j\in \kappa_{g(u)}}\in r_{g(u)}\},
\]
where $g:U\ra K$ and $f:\bigcup_{u\in U}\{u\}\times \kappa_{g(u)}\ra \kappa'\amalg \tilde{\kappa}$
(This condition is referred as ``$r$ is defined from $\{r_k\}_{k\in K}$
by primitive positive formula of $L_{\infty,\infty}$-logic."),
then $r\in R$ holds.
%
\end{df}

\begin{df}[Polymorphism, Invariant Relation]
Let $A$ be a set and $\lambda,\kappa$ be cardinals.
\begin{enumerate}
\item
Let $F$ be a set of $<\mc\lambda$-ary operations on $A$.
A $\kappa$-ary relation $r$ on $A$ is said to be invariant to $F$ if
\[
\forall i\in \lambda'; (a_{ij})_{j\in\kappa}\in r
\ \  \Longrightarrow\ \ 
(f(a_{ij})_{i\in\lambda'})_{j\in\kappa}\in r
\]
hold for all $f\in F$.
The set of all $\kappa$-ary invariant relations of $F$ is denoted by $\Inv_\kappa(F)$.
We define $\Inv_{<\kappa}(F):=\bigcup_{\kappa'<\kappa} \Inv_{\kappa'}(F)$.
\item
Let $R$ be a set of $<\mathclose\kappa$-ary relations on $A$.
A $\lambda$-ary operation $f$ on $A$ is said to be a polymorphism of $R$ if
\[
\forall i\in \lambda; (a_{ij})_{j\in\kappa'}\in r
\ \  \Longrightarrow\ \ 
(f(a_{ij})_{i\in\lambda})_{j\in\kappa'}\in r
\]
hold for all $r\in R$.
The set of all $\lambda$-ary polymorphisms of $R$ is denoted by $\Pol_\lambda(R)$.
We define $\Pol_{<\lambda}(R):=\bigcup_{\lambda'<\lambda} \Pol_{\lambda'}(R)$.
\end{enumerate}
\end{df}

\begin{rem}
For simplifying description, we also use notation such as $\Inv_X(C)$,
where $C$ is a set of operations and $X$ is an arbitrary set.
That is defined as
\[
\Inv_X(C):=\{r\in {\cal P}(A^X)\mid 
  \{(a_i)_{i\in |X|}\mid (a_{\ph(x)})_{x\in X}\in r\}\in \Inv_{|X|}(C)\},
\]
where $\ph$ is a bijection $X\ra |X|$.
This does not depend on the choice of bijection $\ph$.

Similarly, the set of set-indexed polymorphisms is defined as
\[
\Pol_X(R):=\{f:A^X\ra A\mid
 [(a_i)_{i\in |X|}\mapsto f(a_{\ph(x)})_{x\in X}]\in \Pol_{|X|}(R)\}.
\]
\end{rem}

The following proposition is easily follows from definition.
\begin{prop}
Let $A$ be a set and $\lambda,\kappa$ be a cardinal.
\begin{enumerate}
\item
For a set $F$ of $<\mathclose\lambda$-ary operations on $A$, $\Inv_{<\kappa}(F)$ is
a $<\mathclose\kappa$-ary relational clone on $A$.
\item
For a set $R$ of $<\mathclose\kappa$-ary relations on $A$, $\Pol_{<\lambda}(F)$ is
a $<\mathclose\lambda$-ary operational clone on $A$.
\end{enumerate}
\end{prop}


We complete to prepare to describe the correspondence theorem
between clones and relational clones. The theorem is described as follows.
\begin{thm}\label{maincorrespondence}
Let $A$ be a set and $\lambda$ be a strong limit cardinal that satisfies $\lambda>|A|$.
\begin{enumerate}
\item \label{polinv}
If $C$ is a $<\mathclose\lambda$-ary operational clone, then $\Pol_{<\lambda}(\Inv_{<\lambda}(C))=C$ holds.
\item \label{invpol}
If $R$ is a $<\mathclose\lambda$-ary relational clone, then $\Inv_{<\mathclose\lambda}(\Pol_{<\mathclose\lambda}(R))=R$ holds.
\end{enumerate}
\end{thm}
Note that a cardinal $\lambda$ is said to be strong limit if 
$\lambda'<\lambda$ implies $2^{\lambda'}<\lambda$.
The following proof is obtained by basically the same way
as the proof of finitary and on finite set version
described in \cite{Lau} Chapter 2 of Part II.

\begin{proof}
\ref{polinv}. $\Pol_{<\lambda}(\Inv_{<\lambda}(C))\supset C$ is easy.

We prove the reverse inclusion $\Pol_{<\lambda}(\Inv_{<\lambda}(C))\subset C$.
Let $\kappa$ be a cardinal that $\kappa<\lambda$.
Define a $|A|^\kappa$-ary relation $\Gamma_\kappa(C)$ by
\[
\Gamma_\kappa(C):=\{(f(j(i))_{i\in \kappa})_{j\in {A^\kappa}}\mid f\in C_\kappa\},
\] 
Where $C_\kappa$ is the set of all $\kappa$-ary operations belonging to $C$,
i.e., $C_\kappa:=C\cap A^{A^\kappa}$.
Then the next claim holds.
\renewcommand{\labelenumi}{(\arabic{enumi})}
\begin{enumerate}
\item \label{graphisclosed}
$\Gamma_\kappa\in\Inv_{A^\kappa}(C)$.
\item \label{cdefinedgamma}
For any $f:A^\kappa\ra A$, $f\in C_\kappa$ holds if and only if $f$ preserves $\Gamma_\kappa(C)$.
\end{enumerate}
\renewcommand{\labelenumi}{\arabic{enumi}.}

(\ref{graphisclosed}) easily follows from the assumption $C$ is closed under composition.

(\ref{cdefinedgamma}) is proved as follows.
By definition of $\Gamma_\kappa(C)$ and the assumption that $C$ is closed under composition, 
$f\in C_\kappa$ implies $f$ preserves $\Gamma_\kappa(C)$.

To prove the converse, notice that for each $i_0\in\kappa$, 
\[
(j(i_0))_{j\in {A^\kappa}}=(\pi_{i_0}^{A^\kappa}(j(i))_{i\in \kappa})_{j\in {A^\kappa}}
 \in \Gamma_\kappa(C)
\]
holds. Therefore, if $f\not\in C_\kappa$, then $(j(i))_{j\in {A^\kappa}}\in \Gamma_\kappa(C)$
but $(f(j(i))_{i\in \kappa})_{j\in {A^\kappa}}\not\in \Gamma_\kappa(C)$.
That means $f$ does not preserve $\Gamma_\kappa(C)$.

By these claims, we conclude that
\[
f\in\Pol_\kappa(\Inv_{<\lambda}(C))
\ \Rightarrow\ f\in \Pol_\kappa(\Gamma_\kappa(C))
\ \Rightarrow\ f\in C.
\]

\ref{invpol}. $\Inv_{<\lambda}(\Pol_{<\lambda}(R))\supset R$ is easy.

To prove reverse inclusion, first we prove 
$\Gamma_\kappa(\Pol_{<\lambda}(R))\in R$ for arbitrary $\kappa<\lambda$.
Let $r$ be the minimum relation that $r\supset \Gamma_\kappa(\Pol_{<\lambda}(R))$ and $r\in R_{\kappa}$.
We should prove $r=\Gamma_\kappa(\Pol_{<\lambda}(R))$.

Assume $a=(a_j)_{j\in A^\kappa}\in r\sm \Gamma_\kappa(\Pol(R))$ exists.
Then, by Claim (\ref{cdefinedgamma}),
the operation $f_a:j\mapsto a_j$ ($A^\kappa\ra A$) does not belong to $\Pol_{\kappa}(R)$.
Therefore, there exist $s\in R_{\mu}$ ($\mu<\lambda$) and 
$(b_{i,k})_{i\in\kappa,k\in\mu}\in A^{\kappa\mu}$ that satisfy the following conditions:
\begin{itemize}
\item
$(b_{i,k})_{k\in\mu}\in s$ for all $i\in\kappa$.
\item
$(f_a(b_{i,k})_{i\in\kappa})_{k\in\mu}\not\in s$.
\end{itemize}
Let $\ph:\mu\ra A^\kappa$ be the unique mapping such that $b_{i,k}=\ph(k)(i)$ and
define an relation $\tilde{r}$ by
\[
\tilde{r}:=\{(x_j)_{j\in A^\kappa}\mid (x_j)_{j\in A^\kappa}\in r\land 
\exists (y_k)_{k\in\mu}\in s(\bigwedge_{k\in\mu} y_k=x_{\ph(k)})\}.
\]
Then the following assertions hold.
\begin{itemize}
\item
$\tilde{r}\in R$.
\item
$\Gamma_\kappa(\Pol_{<\lambda}(R))\subset \tilde{r}\subset r$.
\item
$a\not\in \tilde{r}$.
\end{itemize}
These properties contradict to minimumity of $r$.
Therefore, $\Gamma_\kappa(\Pol_{<\lambda}(R))=r\in R$ holds if these assertions are proved.

$\tilde{r}\in R$ follows from $r,s\in R$. $\tilde{r}\subset r$ is trivial.

We prove $\Gamma_\kappa(\Pol_{<\lambda}(R))\subset \tilde{r}$.
Note that $\tilde{r}$ is closed under operations belonging to \mbox{$\Pol_{<\lambda}(R)\subset \Pol_{<\lambda}(\tilde{r})$}.
Adding this and the fact that $\Gamma_\kappa(\Pol_{<\lambda}(R))$ is the minimum relation that contains 
\mbox{$\{(j(i_0))_{j\in A^\kappa}\mid i_0\in \kappa\}$} and closed under every operations belong to $\Pol_{<\lambda}(R)$, 
it is sufficient to show that $(j(i_0))_{j\in A^\kappa}\in \tilde{r}$ for all $i_0\in\kappa$.
It follows from
\[
(j(i_0))_{j\in A^\kappa}\in \Gamma_\kappa(\Pol_{<\lambda}(R))\subset r\text{ and }(\ph(k)(i_0))_{k\in\mu}=(b_{i_0,k})_{k\in\mu}\in s.
\]
Finally, we prove $a\not\in r$. It follows from 
\[
(f_a(b_{i,k})_{i\in\kappa})_{k\in\mu}=(f_a(\ph(k)(i))_{i\in\kappa})_{k\in\mu}
=(a_{\ph(k)})_{k\in\mu}\not\in s.
\]
The proof of $\Gamma_\kappa(\Pol_{<\lambda}(R))\in R$ is completed.

Next, we prove $\Inv_{<\lambda}(\Pol_{<\lambda}(R))\subset R$.
Let $r\in\Inv_\kappa(\Pol_{<\lambda}(R))$ and $r=\{(a_{i,j})_{j\in\kappa}\}_{i\in I}$ be an enumeration of $r$.
Let $\ph:\kappa\ra A^I$ be the unique mapping satisfying $\ph(j)(i)=a_{i,j}$.
Define 
\[
\tilde{r}:=\{ (x_j)_{j\in\kappa}\in A^\kappa \mid
  \exists(y_k)_{k\in A^I}\bigwedge_{j\in\kappa}( (y_k)_{k\in A^I} \in \Gamma_I(\Pol_{<\lambda}(R))
          \land y_{\ph(j)}=x_j)\} .
\]
Clearly $\tilde{r}\in R$ holds.
We prove $r=\tilde{r}$.
For $(a_{i_0,j})_{j\in\kappa}$, $(y_k)_{k\in A^I}=(k(i_0))_{k\in A^I}\in A^{A^I}$ satisfies
\[
(y_k)_{k\in A^I}\in \Gamma_I(\Pol_{<\lambda}(R)) \text{ and }y_{\ph(j)}=a_{i_0,j}.
\]
Therefore $(a_{i_0,j})\in\tilde{r}$ holds and $r\subset \tilde{r}$ is proved.

Finally, we prove $\tilde{r}\subset r$.
Because $\Gamma_I(\Pol_{<\lambda}(R))$ is generated (as a $\Pol_{<\lambda}(R)$-algebra) by \mbox{$\{(k(i))_{k\in A^I}\mid i\in I\}$},
the set $\{(\ph(j)(i))_{j\in \kappa}\mid i\in I\}=r$ is a set of generator of $\tilde{r}$.
(Because $\{(\ph(j)(i))_{j\in \kappa}\mid i\in I\}$ is
the image of $\{(k(i))_{k\in A^I}\mid i\in I\}$ by the projection $A^{A^I}\ra A^\kappa$,
which is a homomorphism between $\Pol_{<\lambda}(R)$-algebras, induced by $\ph:\kappa\ra A^I$,
and $\tilde{r}$ is the image of $\Gamma_I(\Pol_{<\lambda}(R))$ by the same projection.)
By this fact and $r$ is closed under operations belonging to $\Pol_{<\lambda}(R)$, $\tilde{r}\subset r$ holds.
\end{proof}

\section{Characterization of finitary clones}
In this section, we describe the relational clone corresponding to 
the clone of all finitary operations.
That is the set of all ``generalized diagonal relations" defined as follows.
\begin{df}
Let $A$ be a set, $\kappa\geq |A|$ be an infinite cardinal.
$D^\kappa_{{\rm fin}}$ denotes the following $\kappa$-ary relational clone:
\begin{itemize}
\item
For a set ${\cal E}$ of equivalence relations on $\kappa$, we define
$D_{{\cal E}}:=\{(a_i)_{i\in\kappa}\in A^\kappa \mid \{(i,j)\mid a_i=a_j\}\in{\cal E}\}$.
\item
$D_{{\rm fin}}^\kappa:=\{D_{{\cal E}}\mid {\cal E}\text{ is an ideal of the lattice of all equivalence relations on }\kappa\}$.
\end{itemize}
\end{df}

\begin{prop}\label{finitaryoperation}
Let $A$ be a set, $\lambda,\kappa$ be infinite cardinals and $f:A^\lambda \ra A$.
Assume $|A|,\lambda\leq\kappa$.
Then the following conditions are equivalent.
\begin{enumerate}
\item
$f$ is essentially finitary, namely
there exists a finite set $I\subset \lambda$ such that for any $(a_i)_{i\in\lambda},(b_i)_{i\in\lambda}\in A^\lambda$,
$(a_i)_{i\in I}=(b_i)_{i\in I}$ implies $f(a_i)_{i\in\lambda}=f(b_i)_{i\in\lambda}$.
\item
$f\in\Pol_\lambda D^\kappa_{{\rm fin}}$.
\end{enumerate}
\end{prop}
\begin{proof}
1 $\Rightarrow$ 2. Let $f:A^\lambda\ra A$ be essentially finitary, depends on finite components $I\subset \lambda$.
Let ${\cal E}$ be an ideal of the lattice of equivalence relations on $\kappa$.
Suppose $(a_{ij})_{j\in\kappa}\in D_{{\cal E}}$ for $i\in \lambda$,
namely $\{(j_1,j_2)\mid a_{ij_1}=a_{ij_2}\}\in{\cal E}$ for all $i\in\lambda$.
Particularly, $\{(j_1,j_2)\mid a_{ij_1}=a_{ij_2}\}\in{\cal E}$ hold for all $i\in I$.
Because ${\cal E}$ is an ideal and $I$ is finite,
$E:=\bigcap_{i\in I}\{(j_1,j_2)\mid a_{ij_1}=a_{ij_2}\}\in{\cal E}$.
For each pair $(j_1,j_2)\in E$,
$f(a_{ij_1})_{i\in\lambda}=f(a_{ij_2})_{i\in\lambda}$ holds.
That means $(f(a_{ij})_{i\in\lambda})_{j\in\kappa}\in D_{{\cal E}}$.

2 $\Rightarrow$ 1. Suppose $f$ is not essentially finitary.
Let $\mu$ be the minimum cardinal that satisfies the following condition:
There is a set $I\subset \lambda$ that $|I|=\mu$ and the implication
\begin{equation}\label{essentialvariables}
a_i=b_i\text{ (for all }i\in I\text{) }
\Longrightarrow\  f(a_i)_{i\in\lambda}=f(b_i)_{i\in\lambda}
\end{equation}
holds.
Let $I\subset \lambda$ be a set satisfying $|I|=\mu$ and Implication (\ref{essentialvariables}).
Since $f$ is not essentially finitary, $\mu$ is an infinite cardinal.
Let $\alpha$ be the minimum ordinal that has the cardinality $\mu$
and fix a bijection 
$i:\alpha\ra I$ $(\beta\mapsto i_\beta)$.

By the definition and assumption for $\mu,\alpha,I$ and the mapping $i$, 
there are 
tuples $(a_{i,\beta})_{i\in \lambda,\beta\in \alpha},
(b_{i,\beta})_{i\in \lambda,\beta\in \alpha}$
of elements of $A$ satisfying the following conditions:
\begin{itemize}
\item
If $i\in \{i_\gamma\mid \gamma\leq \beta\}$ then $a_{i,\beta}=b_{i,\beta}$.
\item
$f(a_{i,\beta})_{i\in\lambda}\neq f(b_{i,\beta})_{i\in\lambda}$ hold for any $\beta<\alpha$.
\end{itemize}

Let $k:\alpha\times \{0,1\}\hookrightarrow \kappa$ $((\beta,e)\mapsto k_{\beta,e})$
be an injection.
Define
\[
{\cal E}:=\{E\mid E\text{ is an equivalence relation on }\kappa, 
\exists \beta<\alpha; \beta<\gamma<\alpha\Rightarrow (k_{\gamma,0},k_{\gamma,1})\in E\}.
\]
Then ${\cal E}$ is an ideal of equivalence relations on $\kappa$.
Fix an element $c_0\in A$ and
define $(c_{i,k})_{i\in\lambda,k\in\kappa}\in A^{\lambda\kappa}$ as follows:
\[
c_{i,k}:=
\begin{cases}
a_{i,\beta}& \text{ if there is }\beta \text{ such that }k=k_{\beta,0},\\
b_{i,\beta}& \text{ if there is }\beta \text{ such that }k=k_{\beta,1},\\
c_0& \text{ otherwise}.
\end{cases}
\]
Then $(c_{ik})_{k\in\kappa}\in D_{{\cal E}}$ for each $i\in\lambda$. 
However $(f(c_{i,k})_{i\in\lambda})_{k\in\kappa}\not\in D_{{\cal E}}$.
It means $f\not\in \Pol(D_{{\cal E}})$.
\end{proof}

As a corollary of this proposition, we obtain correspondence between finitary clones
and relational clones that contain all generalized diagonal relations.
\begin{cor}\label{finitaryclones}
Let $A$ be a set, $\lambda$ be a power limit cardinal that $\lambda>|A|$.
Then $\Pol_{<\lambda}$ and $\Inv_{<\lambda}$ are mutually inverse mapping between
the set of all essentially finitary clones on $A$ and relational clones
that contain $D^{<\lambda}_{{\rm fin}}:=\bigcup_{\lambda'<\lambda} D^{\lambda'}_{{\rm fin}}$.
\end{cor}
\begin{proof}
If $<\mc\lambda$-ary clone $C$ only contains essentially finitary operations,
then $\Inv_{<\lambda}(C)\supset D^{<\lambda}_{{\rm fin}}$ holds by the previous theorem.

To prove the converse, suppose a $<\mc\lambda$-ary clone $C$ contains
an operation $f$ not essentially finitary.
Then, by the previous theorem, there is an equivalence relation
${\cal E}$ on $\lambda'<\lambda$ that $f$ does not preserve $D_{{\cal E}}$.
Therefore $\Inv_{<\lambda}(C)\not\supset D^{<\lambda}_{{\rm fin}}$.
\end{proof}


\begin{thebibliography}{99}
\bibitem{BKKR}
V. G. Bodnar\v{c}uk, L. A. Kalu\v{z}nin, V. N. Kotov, B. A. Romov: Galois theory for Post algebras, I-II(Russian), Kibernetika 5(1969), 1-10;1-9.
\bibitem{Gei}
D. Geiger: Closed systems of functions and predicates, Pacific J. Math. 27(1968). 95-100.
\bibitem{Lau}
Dietlinde Lau: Function algebras on finite sets. A basic course on many-valued logic and clone theory. Springer Monographs in Mathematics. Springer,  
Berlin, (2006)
\bibitem{Pos}
Reinhard P\"{o}schel: A general Galois theory for operations and
relations and concrete characterization of related algebraic structures,
Akademie der Wissenschaften der DDR Institut fur
Mathematik, Report 1980, vol. 1, Berlin, 1980

\end{thebibliography}
\end{document}